\documentclass[11pt]{article}
\usepackage[a4paper,margin=1.2in]{geometry}

\usepackage[colorlinks,citecolor=magenta,linkcolor=black]{hyperref}
\pdfpagewidth=\paperwidth \pdfpageheight=\paperheight
\usepackage{amsfonts,amssymb,amsthm,amsmath,eucal,tabu,url}
\usepackage{pgf}
 \usepackage{array}
 \usepackage{tikz-cd}
 \usepackage{pstricks}
 \usepackage{pstricks-add}
 \usepackage{pgf,tikz}
 \usetikzlibrary{automata}
 \usetikzlibrary{arrows}
 \usepackage{indentfirst}
\usepackage{tabularx} 

\theoremstyle{plain}
\newtheorem{thm}{Theorem}[section]
\newtheorem{theorem}[thm]{Theorem}

\newtheorem{proposition}[thm]{Proposition}
\newtheorem{corollary}[thm]{Corollary}

\theoremstyle{definition}
\newtheorem{definition}[thm]{Definition}
\newtheorem{remark}[thm]{Remark}
\newtheorem{example}[thm]{Example}

\newtheorem{problem}[thm]{Problem}

\newtheorem{thevarthm}[thm]{\varthmname}

\newenvironment{varthm*}[1]{\trivlist\item[]{\bf #1.}\it}{\endtrivlist}

\renewcommand\geq{\geqslant}

\renewcommand\leq{\leqslant}

\newcommand\be{\begin{eqnarray*}}
\newcommand\ee{\end{eqnarray*}}

\newcommand\newop[2]{\def#1{\mathop{\rm #2}\nolimits}}
\newop\log{log}
\newop\ord{ord}
\newop\Gal{Gal}
\newop\SL{SL}
\newop\Bl{Bl}
\newop\mult{mult}
\newop\mass{mass}
\newop\div{div}
\newop\codim{codim}
\newop\sing{sing}
\newop\vdim{vdim}
\newop\edim{edim}
\newop\Ass{Ass}
\newop\size{size}
\newop\reg{reg}
\newop\satdeg{satdeg}
\newop\supp{supp}
\newop\Neg{Neg}
\newop\Nef{Nef}
\newop\Nefh{Nef_H}
\newop\Eff{Eff}
\newop\Zar{Zar}
\newop\MB{MB}
\newop\MBxC{MB\mathit{(x,C)}}
\newop\NnB{NnB}
\newop\Bigg{Big}
\newop\Effbar{\overline{\Eff}}

\def\keywordname{{\bfseries Keywords}}%
\def\keywords#1{\par\addvspace\medskipamount{\rightskip=0pt plus1cm
\def\and{\ifhmode\unskip\nobreak\fi\ $\cdot$
}\noindent\keywordname\enspace\ignorespaces#1\par}}
\def\subclassname{{\bfseries Mathematics Subject Classification
(2020)}\enspace}
\def\subclass#1{\par\addvspace\medskipamount{\rightskip=0pt plus1cm
\def\and{\ifhmode\unskip\nobreak\fi\ $\cdot$
}\noindent\subclassname\ignorespaces#1\par}}

\begin{document}
\title{Supersolvable resolutions of line arrangements }
\author{Jakub Kabat}
\date{\today}
\maketitle
\thispagestyle{empty}
\begin{abstract}
The main purpose of the present paper is to study the numerical properties of supersolvable resolutions of line arrangements. We provide upper-bounds on the so-called extension to supersolvability numbers for certain extreme line arrangements in $\mathbb{P}^{2}_{\mathbb{C}}$ and we show that these numbers \textbf{are not} determined by the intersection lattice of the given arrangement. 
\subclass{52C35, 14N20, 13D45}
\end{abstract}

\section{Introduction}
The present note is strongly motivated by an interesting article due to Ziegler \cite{Ziegler}, where he studied Terao's conjecture in arbitrary characteristics. This conjecture states that the freeness of the module of logarithmic vector fields of an arrangement should depend only on the intersection poset. This conjecture is widely open and difficult to verify. In most cases, interesting examples of free hyperplane arrangements are rigid, i.e., the arrangements are projectively uniquely, defined up to an automorphism of ${\rm PGL}(n+1,\mathbb{K})$, and because of that they cannot lead to potential counterexamples to Terao's conjecture. In order to study the aforementioned conjecture, Ziegler introduced the notion of supersolvable resolutions of arrangements. The key advantage of this construction is that it depends on the embedding of the arrangement and thus on the specific representation of the intersection poset.
\begin{definition}
Let $\mathcal{A} \subset \mathbb{P}^{n}_{\mathbb{K}}$ be an arrangement of hyperplanes and denote by $L(\mathcal{A})$ the intersection lattice of $\mathcal{A}$. We say that $\mathcal{A}$ is supersolvable if $L(\mathcal{A})$ is supersolvable as a lattice.
\end{definition}
As proved by Jambu and Terao in \cite{JT}, supersolvable hyperplane arrangements are free and their freeness is determined by the combinatorics. This is the main reason why supersolvable hyperplane arrangements play an important role in the world of free arrangements. Based on that, we can formulate the main definition of the present note.
\begin{definition}
Let $\mathcal{A} \subset \mathbb{P}^{n}_{\mathbb{K}}$ be a hyperplane arrangement. A supersolvable resolution of $\mathcal{A}$ is a finite sequence of arrangements
$$Y_{\bullet}: \quad X = Y_{0} \subseteq Y_{1} \subseteq  ... \subseteq Y_{k} = Y$$
such that $|Y_{i}| = |\mathcal{A}| + i$ for $0
\leq i \leq k$ and $Y$ is supersolvable. We will denote, in more specific situations, the resulting arrangement $Y$ as $\mathcal{A}^{RS}$.
\end{definition}
 It is very natural to wonder to what extent we can find supersolvable resolutions of arrangements efficiently, i.e., when the chains $Y_{\bullet}$ are essentially the shortest possible. This motivates the following definition.
\begin{definition}
Let $\mathcal{A} \subset \mathbb{P}^{n}_{\mathbb{K}}$ be an arrangement of hyperplanes. Then we define the extension to supersolvability number of $\mathcal{A}$ as 
$${\rm ext SS}(\mathcal{A}) = {\rm min} \{ k : Y_{\bullet} \text{ is a supersolvable resolution of } \mathcal{A} \text{ and  }Y_{k} \text{ is supersolvable}\}.$$

In the present paper we focus on the case of line arrangements in the complex projective plane and our main goal is to provide some numerical results on ${\rm extSS}$ numbers of certain line arrangements. It is worth pointing out that this setting is not very restrictive since even in that case it is very difficult to compute the actual values of ${\rm ext SS}(\mathcal{A})$ in general. 

\end{definition}
\begin{remark}(cf. \cite[Lemma 3.3]{Ziegler})
Let $\mathcal{A} \subset \mathbb{P}^{2}_{\mathbb{C}}$ be an arrangement of lines (we may assume that $\mathcal{A}$ is not supersolvable), then the number ${\rm ext SS}(\mathcal{A})$ is well-defined and finite.
\end{remark}
\begin{proof}
Let $\mathcal{A}$ be an arbitrary line arrangement and $P$ a general point in the plane. Let $\mathcal{B}$ be the arrangement consisting of all lines in $\mathcal{A}$ and lines joining $P$ with all intersection points of $\mathcal{A}$. Then $\mathcal{B}$ is supersolvable.
\end{proof}
It follows immediately that for any arrangement of lines $\mathcal{A}$ one has
$${\rm ext SS}(\mathcal{A}) \leq |{\rm Sing}(\mathcal{A})|.$$
Note that this upper bound is very rough. For example, if $\mathcal{A}$ is already supersolvable, then ${\rm ext SS}(\mathcal{A}) = 0$. It is thus natural to state the following fundamental question.
\begin{problem}
	Find ${\rm ext SS}(\cdot)$ numbers for relevant line arrangements in $\mathbb{P}^{2}_{\mathbb{C}}$.
\end{problem}
Our investigations on ${\rm extSS}$ numbers start from certain B\"or\"oczky's arrangements of lines. These are line arrangements defined over the reals which have the maximal possible number of triple intersection points, according to Green-Tao's result. We point out, by giving an explicit example of two line arrangement which have the same weak combinatorics, that the problem of finding ${\rm extSS}$ numbers is not a lattice dependent task, and this is the reason why we can provide only upper bounds for ${\rm extSS}$'s in the course of the paper. Furthermore, we present estimates on the ${\rm extSS}$ numbers for classical reflection line arrangements, namely the Klein and the Wiman arrangements of lines. In the last section, we discuss a possible application of our work towards the theory of unexpected curves in the complex projective plane.

\section{Preliminaries}
 Let $\mathbb{K}$ be any field. Consider $\mathcal{A} = \{\ell_{1}, ..., \ell_{d}\} \subset \mathbb{P}^{2}_{\mathbb{K}}$ an arrangement of $d\geq 3$ lines. For each line $\ell_{i}$ we choose a linear form $\alpha_{i} \in S:=\mathbb{K}[x,y,z]$ such that $\ell_{i} = {\rm ker}(\alpha_{i})$. Now we can define the module of $\mathcal{A}$-derivations as
$$D(\mathcal{A}) = \{ \theta \in {\rm Der}_{\mathbb{K}}(S) \, : \, \theta(\alpha_{i}) \in \langle \alpha_{i} \rangle \text{ for all } i \in \{1, ...,d\}\},$$
where ${\rm Der}_{\mathbb{K}}(S) = S\cdot \partial_{x} \oplus S \cdot \partial_{y} \oplus S \cdot \partial_{z}$.
\begin{definition}
We say that an arrangement $\mathcal{A} \subset \mathbb{P}^{2}_{\mathbb{C}}$ is free when $D(\mathcal{A})$ is a free $S$-module. In this case, the degrees of the generators of $D(\mathcal{A})$ are called the exponents.
\end{definition}
Let us recall that the Poincar\'e polynomial of an arrangement $\mathcal{A} \subset \mathbb{P}^{2}_{\mathbb{C}}$ can be defined (you may consider this as a small exercise) by
$$\pi(\mathcal{A},t) = 1 + |\mathcal{A}|\cdot t + \bigg(\sum_{r\geq 2}(r-1)t_{r} \bigg)\cdot t^{2} + \bigg(\sum_{r\geq 2}(r-1)t_{r} + 1 - |\mathcal{A}|\bigg)\cdot t^{3},$$
where $t_{r}$ is the number of $r$-fold points, i.e., points in the plane where exactly $r$ lines from the arrangement meet. If the arrangement $\mathcal{A}$ is free, then by Terao's factorization theorem \cite{Terao} the Poincar\'e polynomial splits over $\mathbb{Z}$ into linear factors, i.e.,
$$\pi(\mathcal{A},t) = \prod_{i=1}^{3}(1+d_{i}\cdot t),$$
and the numbers $(d_{1},d_{2},d_{3})$ are the exponents. Since the Euler derivation $E := x\partial_{x} + y\partial_{y} + z\partial_{z}$ always sits in $D(\mathcal{A})$, then $E$ is always one of the generators of $D(\mathcal{A})$ giving $d_{1} = 1$, and due to this reason it is customarily to call $(d_{2},d_{3})$ the exponents of $\mathcal{A}$.
\section{${\rm extSS}$ numbers for line arrangements in the complex plane}
In this section, we compute the actual values and upper bounds on extensions to supersolvability numbers. We start with a somewhat extreme class of line arrangements in the complex projective plane. We say that $\mathcal{A}_{d} \subset \mathbb{P}^{2}_{\mathbb{C}}$ is general with $d\geq 3$ if the lines are in general position and the only intersection points are double points. 

\begin{proposition}
Let $\mathcal{A} \subset \mathbb{P}^{2}_{\mathbb{C}}$ be an arrangement of $d$ general lines. Then $${\rm ext SS}(\mathcal{A}) \leq \binom{d-2}{2}.$$
\end{proposition}
\begin{proof}
Indeed, let $P$ be any double intersection point of $\mathcal{A}$ and let $\ell, m$ be the lines from $\mathcal{A}$ intersecting at $P$. We denote by $\mathcal{A}' = \mathcal{A} \setminus\{\ell, m\}$. Then we need to join singular points of $\mathcal{A}'$ with $P$. There are exactly $\binom{d-2}{2}$ such points.
\end{proof}

Here is a good moment to explain why we can provide, in general, only upper bounds on the ${\rm extSS}$ numbers.

\begin{example}
We are going to present an example which shows that ${\rm extSS}$ numbers cannot be computed directly from the (weak) combinatorics. Consider the poset $L$ defining the combinatorics of $6$ lines intersecting only at double intersection points. The space of all possible geometric realizations $\mathcal{M}$ of $L$ is extremely large, and this is the key spot in our example. 

If we take an arrangement $\mathcal{L}$ consisting of $6$ general lines with the equations listed below:
\begin{equation*}
\begin{array}{l}
\ell_{1} : x - y + 2z = 0, \\
\ell_{2} : x - y - 2z = 0, \\
\ell_{3} : x + y - 2z = 0, \\
\ell_{4} : x + y + 2z = 0, \\
\ell_{5} : 9x - y + 9z = 0, \\
\ell_{6} : 9x + y -9z = 0.
\end{array}
\end{equation*}

Then one has
$${\rm extSS}(\mathcal{L}) = 6.$$
The claim follows from the fact that there is no triple of double intersection points which lie on a line $\ell$ such that $\ell \not\in \mathcal{L}$. This justifies the equality above. 

Consider now another element $\mathcal{P}$ from the space of realizations depicted on Figure \ref{pap}. This arrangement is constructed with the use of the Pappus theorem, the only difference is that we remove those three lines that are making the intersection points triple in the configuration. By this trick, we obtain an honest arrangement of $6$ lines with only double intersection points. It is easy to see that in this case
$${\rm extSS}(\mathcal{P}) < 6,$$
and it follows from the fact that we have an additional collinearity, denoted by the dashed line, which existence is guarantee by the Pappus theorem.
\begin{figure}[h!]
\centering
\definecolor{qqqqff}{rgb}{0.,0.,1.}
\definecolor{ffqqqq}{rgb}{1.,0.,0.}
\definecolor{ududff}{rgb}{0.30196078431372547,0.30196078431372547,1.}
\begin{tikzpicture}[line cap=round,line join=round,>=triangle 45,x=1.0cm,y=1.0cm,scale=1.1]
\clip(4.633828980663498,-5.4382516190204475) rectangle (18.65398559101832,1.7116991406029973);
\draw [line width=2.pt,domain=4.633828980663498:18.65398559101832] plot(\x,{(--44.-5.*\x)/4.});
\draw [line width=2.pt,domain=4.633828980663498:18.65398559101832] plot(\x,{(--45.-5.*\x)/5.});
\draw [line width=2.pt,domain=4.633828980663498:18.65398559101832] plot(\x,{(--57.-5.*\x)/-3.});
\draw [line width=2.pt,domain=4.633828980663498:18.65398559101832] plot(\x,{(--61.-5.*\x)/1.});
\draw [line width=2.pt,domain=4.633828980663498:18.65398559101832] plot(\x,{(--68.-5.*\x)/-2.});
\draw [line width=2.pt,domain=4.633828980663498:18.65398559101832] plot(\x,{(--65.-5.*\x)/-5.});
\draw [line width=2.pt,dash pattern=on 2pt off 2pt,color=ffqqqq,domain=4.633828980663498:18.65398559101832] plot(\x,{(--0.4761904761904816-0.4761904761904763*\x)/2.3809523809523796});
\begin{scriptsize}
\draw [fill=ududff] (8.,1.) circle (2.5pt);
\draw [fill=ududff] (12.,1.) circle (2.5pt);
\draw [fill=ududff] (14.,1.) circle (2.5pt);
\draw [fill=ududff] (9.,-4.) circle (2.5pt);
\draw [fill=ududff] (12.,-4.) circle (2.5pt);
\draw [fill=ududff] (13.,-4.) circle (2.5pt);
\draw [fill=ududff] (10.285714285714286,-1.8571428571428572) circle (2.5pt);
\draw [fill=ududff] (11.,-2.) circle (2.5pt);
\draw [fill=ududff] (12.666666666666666,-2.3333333333333335) circle (2.5pt);
\end{scriptsize}
\end{tikzpicture}
\caption{A Pappus realization of the combinatorics of $d=6$ lines and $15$ double points.}
\label{pap}
\end{figure}
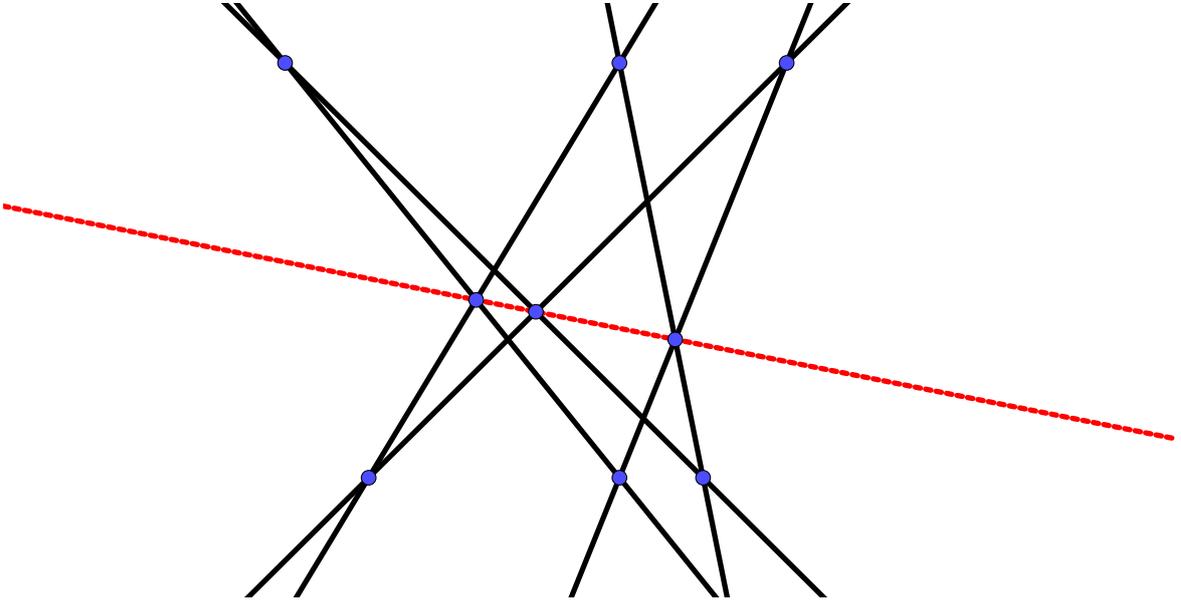

It means that the problem of calculating ${\rm extSS}$ numbers for arrangements does not depend exclusively on the intersection lattice, i.e., if the parameter space of a given arrangement $\mathcal{A}$ (which is nothing else as a space of geometric realizations) is positive dimensional, then we can only hope to find reasonable upper-bounds for ${\rm extSS}(\mathcal{A})$ based on the weak combinatorics.
\end{example}

Now we pass to B\"or\"oczky arrangements of lines. The construction goes as follows. Consider a regular $2n$-gon inscribed in the unit circle in the real affine plane.
   Let us fix one of the $2n$ vertices and denote it by $Q_0$.
   By $Q_{\alpha}$ we denote the point arising by the rotation of $Q_0$ around the center of the circle by angle $\alpha$.

   Then we take the following set of lines
$$\mathcal{B}_{n}=\left\{Q_{\alpha}Q_{\pi - 2\alpha}, \text{ where } \alpha= \frac{2k \pi}{n} \textrm{ for } k=0, \dots, n-1\right\}.$$
   If $\alpha \equiv (\pi - 2\alpha)({\rm mod}\; 2\pi)$,
   then the line $Q_{\alpha}Q_{\pi - 2\alpha}$ is the tangent to the circle at the point $Q_{\alpha}$.
   The arrangement $\mathcal{B}_{n}$ has $\big\lfloor \frac{n(n-3)}{6}\big\rfloor+1$ triple points by \cite[Property 4]{FuPa1984}, and exactly $n-3 + \varepsilon(n)$ double intersection points, where $\varepsilon(n)$ is equal to either $2$ or $0$, depending on the divisibility of $n$. 
   Let us denote the set of these triple points by $\mathbb{T}_n$.
   
Now we recall a simple fact concerning the distribution of triple points    on the arrangement lines.
\begin{proposition}\label{prop:triple point on B lines}
   Every line in the $\mathcal{B}_n$ arrangement contains at least $\big\lfloor\frac{n-3}{2}\big\rfloor$ triple points
   and there exists a line containing at least one more triple point.
\end{proposition}
\proof
   By construction the triple points are distributed on the arrangement lines almost uniformly,
   that means that the difference between the number of points from $\mathbb{T}_n$ on two arrangement lines
   is at most $1$. Let $s$ be the minimal number of triple points on an arrangement line.
   Then it must be
   $$\frac{sn}{3}\leq 1 +\big\lfloor \frac{n(n-3)}{6}\big\rfloor\;\mbox{ and }\;
      \big\lfloor \frac{n(n-3)}{6}\big\rfloor\leq \frac{(s+1)n}{3}$$
   and the claim follows.
\endproof
   From the above result we can also derive the following consequence of Proposition \ref{prop:triple point on B lines}, which is interesting on its own right.
\begin{corollary}\label{cor:alpha for Bn}
   For a fixed $n\geq 8$ let $C$ be a plane curve (possibly reducible and non-reduced) of degree $d$ passing through every point
   in the set $\mathbb{T}_n$ with multiplicity at least $3$. Then $d\geq n$. Moreover, if $d=n$, then $C$ is the union of
   all arrangement lines in $\mathcal{B}_n$.
\end{corollary}
\proof
   Assume to the contrary, that $d<n$.
   By Proposition \ref{prop:triple point on B lines} an arrangement line $\ell$ contains at least $\big\lfloor\frac{n-3}{2}\big\rfloor$
   triple points. If $\ell$ is not a component of $C$, then it must be, by B\'ezout Theorem,
   $$n>d\geq 3\big\lfloor\frac{n-3}{2}\big\rfloor.$$
   It follows that
   $$n+3>3\lfloor\frac{n}{2}\rfloor,$$
   which contradicts the assumption $n\geq 8$.
\endproof
Using some particular symmetries of B\"or\"oczky arrangements of $n=6k$ lines with $k\geq 2$ we can show the following result. Observe in the meantime that for $k=1$ our arrangement $\mathcal{B}_{6}$ is supersolvable.
\begin{theorem}
Let $n=6k$ for $k\geq 2$. Then
$${\rm ext SS}(\mathcal{B}_{6k}) \leq 6k^{2} - 6k.$$
In our construction, the supersolvable resolution  $\mathcal{B}_{6k}^{RS}$ have $6k^2$ lines and the following combinatorics:
$$t_{3 + 6k^{2} - 6k } = 1, \quad t_{4} = 6(k-1)^{2}, \quad t_{3} = 15k-12, \quad t_{2} = 36k^{3}-72k^{2}+42k-3.$$
Finally, the exponents of free arrangement $\mathcal{B}_{6k}^{RS}$ are $d_{1} = 6k-3$, $d_{2} = 6k^{2}-6k+2$.
\end{theorem}
\begin{proof}
Here we present a detailed sketch of our construction. Take one of the points of multiplicity $3$ of $\mathcal{B}_{6k}$ and denote this point by $O$. Take the three lines passing through $O$. Observe that each of the three lines contains exactly $3k$ singular points from the arrangement. Since the only intersection point of the three lines is $O$, then on the three lines we have altogether exactly $9k-2$ intersection points, among them exactly $3$ double points. Now we construct our extension $\mathcal{B}_{6k}^{RS}$ by joining $O$ with each singular point except those $9k-2$ lying on the three lines. Simple calculation tells us that we add the following number of lines
$$6k-3 + \frac{6k(6k-3)}{6} + 1 - (9k - 2) = 6k-3 + 6k^2 - 3k + 1 - 9k + 2 = 6k^{2}-6k.$$
Now we can describe the combinatorics of $\mathcal{B}_{6k}^{RS}$. By the construction, the vertex $O$ has multiplicity $6k^{2}-6k+3$, and this is the only point of such multiplicity. Next, we obtain quadruple points by joining the previous triple points with the vertex $O$, we have altogether
$$t_{4} = 6k^{2} - 3k + 1 - (9k - 5) = 6k^{2} - 12k + 6 = 6(k-1)^{2}.$$
We get also new triple point (out of old double points), there are exactly $6k-6$ such points. Altogether we have
$$t_{3} = 6k - 6 + 9k - 5 - 1 = 15k - 12,$$
where the last $-1$ in the middle equality comes from the fact that $O$ is no longer a triple point.
Finally, we can compute the number of double points.
Using the combinatorial count we obtain that
$$t_{2} = \frac{6k^{2}(6k^{2}-1)}{2} - 3\cdot (15k-12) - 6 \cdot(6(k-1)^{2}) - \binom{6k^{2}-6k+3}{2} = 36k^{3} - 72k^{2}+42k-3.$$
By the construction, $\mathcal{B}_{6k}^{RS}$ is supersolvable and by Jambu-Terao's result \cite{JT}, the arrangement is free. We compute the exponents of the arrangement with use of the Poincar\'e polynomial. Observe that
$$\pi(\mathcal{B}_{6k}^{RS};t) = (1+t)\bigg(1 + (6k^{2}-1)t + (36k^{3}-54k^{2}+30k-6)t^{2}\bigg).$$
Since $\triangle_{t} = (6k^{2}-12k+5)^{2}$ and $6k^{2}-12k+5$ is non-negative for $k\geq 2$, we can compute rational roots of the polynomial, namely
$$a_{1} = \frac{-6k^{2}+1 + 6k^{2} - 12k + 5}{12(2k-1)(3k^{2}-3k-1)} = \frac{-(2k-1)}{2(2k-1)(3k^{2}-3k+1)} = \frac{-1}{6k^{2}-6k+2},$$
$$a_{2} = \frac{-6k^{2}+1 - 6k^{2} + 12k - 5}{12(2k-1)(3k^{2}-3k-1)} = \frac{-12k^{2}+12k-4}{12(2k-1)(3k^{2}-3k+1)} = \frac{-1}{3(2k-1)}.$$
This gives us finally that
$$\pi(\mathcal{B}_{6k}^{RS};t) = (1+t)(1+(6k-3)t)(1+(6k^{2}-6k+2)t),$$
and the exponents are $d_{2} =6k-3$, $d_{3} = 6k^{2}-6k+2$.
\end{proof}

Now we turn to the Klein arrangement of lines $\mathcal{K}$ (see \cite{Klein}). Let us recall that the arrangement $\mathcal{K}$ consists of $d=21$ lines and $t_{3}=28$, $t_{4}=21$.
\begin{proposition}
For the Klein arrangement of lines $\mathcal{K}$ we have ${\rm ext SS}(\mathcal{K})\leq 20$.
\end{proposition}
\begin{proof}
Each line from the arrangement contains exactly $4$ triple and $4$ quadruple singular points. Choose one of the quadruple points and the four lines passing through it. Denoting this quadruple point by $O$, we observe that these four lines contain exactly $4\cdot 8 - 3 = 29$ singular points, so we are left with $12$ triple points and $8$ quadruple points. Next, we join each of the remaining $20$ singular points with $O$, so altogether our line arrangement $\mathcal{K}^{RS}$ consists of $21+20 = 41$ lines and the following intersection points:
$$t_{24} = 1, \quad t_{5} = 8, \quad t_{4} = 24, \quad t_{3}=16, \quad t_{2} = 272.$$
By the construction, $\mathcal{K}^{RS}$ is supersolvable, and we can compute the exponents. Observe that
$$\pi(\mathcal{K}^{RS};t) = (1+t)\bigg( 1+40t + 391t^{2}\bigg) = (1+t)(1+17t)(1+23t),$$
so the exponents are $d_{2} = 17$, $d_{3} = 23$.

\end{proof}
Finally, we consider the last arrangement of our interests, namely the Wiman arrangement of lines \cite{Wiman96}, denoted by $\mathcal{W}$. This remarkable arrangement consists of $45$ lines and it has
$$t_{3}=120, \quad t_{4} = 45, \quad t_{5}=36.$$
\begin{proposition}
For the Wiman arrangement of lines $\mathcal{W}$ one has ${\rm ext SS}(\mathcal{W})\leq 125.$
\end{proposition}
\begin{proof}
Let us recall the most crucial fact about the singular points of Wiman's arrangement of lines. We observed that each line from the arrangement contains exactly $4$ quintuple, $4$ quadruple, and $8$ triple singular points. Choose one of the quintuple points and the five lines passing through this point.  Denoting this point by $O$, we observe that these five lines contain exactly $76$ singular points, so we are left with $80$ triple points, $25$ quadruple points, and $20$ quintuple points.  Next, we join each of the mentioned singular points with $O$, so altogether our line arrangement $\mathcal{W}^{RS}$ consists of $45 + 125 = 170$ lines and it has the following intersection points:
$$t_{130} = 1, \quad t_{6} = 20, \quad t_{5} = 40, \quad t_{4}=100, \quad t_{3} = 40, \quad t_{2} = 4560.$$
By the construction, $\mathcal{W}^{RS}$ is supersolvable, and we can compute the exponents. Observe that
$$\pi(\mathcal{W}^{RS};t) = (1+t)\bigg( 1+ 169t + 5160t^{2}\bigg) = (1+t)(1+40t)(1+129t),$$
so the exponents are $d_{2} = 40$, $d_{3} = 129$.
\end{proof}

\section{Supersolvability and unexpected curves}

In the last section, let us present another motivation that leads us to study the mentioned extensions to the supersolvability property. Very recently, the theory of unexpected curves has appeared and gained a lot of attention by researchers.

Let $\mathcal{P} = \{P_{1}, ..., P_{s}\} \subset \mathbb{P}^{2}_{\mathbb{C}}$ be a finite set of points and let $m_{1}, ..., m_{s}$ be the multiplicities of $\mathcal{P}$. Denote by $X = m_{1}P_{1} + ... + m_{s}P_{s}$ a fat point scheme and consider the associated ideal
	$$I(X) = \bigcap_{i=1}^{s} I(P_{i})^{m_{i}}.$$
	Now we can define the expected dimension by
$${\rm expdim}\, I(X)_{d} = \max \bigg\{ \binom{d+2}{2} -\sum_{i=1}^{s} \binom{m_{i}+1}{2} , 0 \bigg\},$$
where by $I(X)_{d}$ we mean the homogeneous component of degree $d$.
Geometrically speaking, the vector space $I(X)_{d}$ is the linear system of plane curves of degree $d$ passing through each point $P_{i}$'s with multiplicity at least $m_{i}$. The expected dimension informs us whether we can expect the existence of such curves, and in general ${\rm dim} \, I(X)_{d} \geq {\rm expdim} \,I(X)_{d}$.

\begin{definition}
Let $d$ be a non-negative integer. We say that a finite set of points $Z$ in the complex projective plane admits an unexpected curve in degree $d$ with a general point $P$ of multiplicity $d-1$ if
$${\rm dim} (I(Z + (d-1)P))_{d} > {\rm max }\bigg\{ {\rm dim} \, I(Z)_{d} - \binom{d}{2}, 0 \bigg\}.$$
\end{definition}
\begin{definition}
	We say that an arrangement of lines $\mathcal{L} \subset \mathbb{P}^{2}_{\mathbb{C}}$ admits an unexpected curve if the set of points $Z$ dual to the configuration of lines in  $\mathcal{L}$ admits an unexpected curve.
\end{definition}
In the context of supersolvable line arrangements, Di Marca, Malara, and Oneto \cite{diMarca} proved the following result.

\begin{theorem}
A supersolvable line arrangement $\mathcal{L}$ admits an unexpected curve if and only if $d > 2m$ where $d$ is the number of lines and $m$ is the maximum multiplicity of an intersection point of the lines in $\mathcal{L}$.
\end{theorem}
This theorem provides as a very nice criterion for the existence of unexpected curves and once we are able to extend a well-known arrangement in such a way that the resulting object is supersolvable and satisfies the condition that $d>2m$, then we have a new example of an unexpected curve. Let us start with a baby-case of Fermat arrangements of lines.
\begin{example}
Fermat arrangement of lines $\mathcal{F}_{n}$ is defined in the complex projective plane by the linear factors of the following polynomial
$$F(x,y,z)=(x^{n} -y^{n})(y^{n}-z^{n})(z^{n}-x^{n}),$$
where $n\geq 3$.
It is well-known that the arrangement consists of $3n$ lines and $t_{n}=3$, $t_{3}=n^{2}$. It is easy to observe that this is not a supersolvable arrangement since the three fundamental points (which are the intersection points) cannot be joined by lines from the arrangements. One of the smallest extensions of Fermat arrangements looks as follows:
$$\widetilde{F}(x,y,z)=xy(x^{n} -y^{n})(y^{n}-z^{n})(z^{n}-x^{n}),$$
where as previously $n\geq 3$. The new arrangement consists of $3n+2$ lines and
$$t_{2} = 2n, \quad t_{3}=n^{2}, \quad t_{n+1}=2, \quad t_{n+2}=1.$$ This arrangement is clearly supersolvable and the exponents are $d_{1} = n+1$, $d_{2} = 2n$. Moreover, we see that $3n+2 = d > 2m = 2n+4$, since $n\geq 3$, so our new family of line arrangements, denoted in the literature by $\mathcal{A}_{3}^{2}(n)$, leads to new examples of unexpected curves. Moreover, one has
$${\rm extSS}(\mathcal{F}_{n}) = 2$$
provided that $n\geq 3$. Please note for $n=2$ the arrangement $\mathcal{F}_{2}$ is combinatorially equivalent to $\mathcal{B}_{6}$, thus supersolvable and  ${\rm extSS}(\mathcal{F}_{2}) = 0$.

\end{example}
\section*{Acknowledgments}
The author was partially supported by the National Science Centre, Poland, Preludium Grant \textbf{UMO 2018/31/N/ST1/02101}. 

\vskip 0.5 cm

\bigskip
Jakub Kabat,
Department of Mathematics,
Pedagogical University of Krakow,
ul. Podchorazych 2,
PL-30-084 Krak\'ow, Poland. \\
\nopagebreak
\textit{E-mail address:} \texttt{jakub.kabat@up.krakow.pl}

\end{document}